\documentclass[12pt,a4paper]{article}

\usepackage[a4paper, top=2.5cm, bottom=2.5cm, left=2.5cm, right=2.5cm]{geometry}

\usepackage[utf8]{inputenc}
\usepackage[T1]{fontenc}

\usepackage{amsmath,amssymb,amsthm,mathtools}
\usepackage{bm}

\usepackage{algorithm}
\usepackage{algpseudocode} 

\usepackage{booktabs}
\usepackage{multirow}
\usepackage{siunitx}

\usepackage{enumitem}

\usepackage{graphicx}
\usepackage{subcaption}
\usepackage{float} 
\graphicspath{{./figures/}{./Figure/}} 

\usepackage{authblk}

\usepackage{hyperref}
\hypersetup{
    colorlinks=true,
    linkcolor=blue,
    filecolor=magenta,
    urlcolor=cyan
}

\newcommand{\R}{\mathbb R}

\newcommand{\E}{\mathbf {E}}
\newcommand{\Var}{\mathbf{Var}}

\theoremstyle{definition}
\newtheorem{definition}{Definition}[section]
\newtheorem{remark}[definition]{Remark}

\theoremstyle{plain}
\newtheorem{theorem}{Theorem}[section]

\newtheorem{proposition}[theorem]{Proposition}

\usepackage{xifthen}
\newcommand{\textcite}[2][]{\ifthenelse{\isempty{#1}}{\cite{#2}}{\cite[#1]{#2}}}
\newcommand{\parencite}[2][]{\ifthenelse{\isempty{#1}}{\cite{#2}}{\cite[#1]{#2}}}

\usepackage{xcolor}
\usepackage{soul}
\setstcolor{red}

\title{A New Estimator of Kullback--Leibler Divergence via Shannon Entropy}

\author[1]{\footnotesize Mehmet S\i dd\i k \c{C}ad\i rc\i}
\author[2*]{\footnotesize Martin Singull}
\affil[1]{\footnotesize Faculty of Science, Department of Statistics, Cumhuriyet University, Sivas, T\"urkiye}
\affil[2]{\footnotesize Department of Mathematics, Link\"oping University, Link\"oping, Sweden}
\affil[*]{\footnotesize Corresponding author: \href{mailto:martin.singull@liu.se}{\texttt{martin.singull@liu.se}}}
\date{}

\begin{document}
\maketitle

\begin{abstract}
We examine the estimation of the Kullback-Leibler (KL) divergence and the use of the goodness-of-fit test for multivariate continuous distributions. Our starting point is the maximum entropy principle for Shannon entropy: among all distributions with a fixed mean vector and covariance matrix, the multivariate Gaussian distributions uniquely maximize entropy. As a result, the KL divergence from a moment-matched Gaussian distribution to an unknown density can then be written as the \emph{entropy difference}, which is a suitable information-theoretic measure of divergence from the Gaussian distribution.

To estimate, we use $k$-nearest neighbor (kNN) estimators based on Shannon entropy and KL divergence derived from the Kozachenko-Leonenko approach and subsequent improvements, along with the consistency and $L^{2}$-convergence results established for these estimators. Motivated by previous entropy-based goodness-of-fit ideas developed for Rényi-type functionals under generalized Gaussian and Student-type models, we describe a KL-based test statistic as being the difference between (i) the entropy of a Gaussian model fitted to the sample mean and covariance and (ii) the KL divergence between the unknown entropy and the kNN estimate. The statistic converges to zero under multivariate normality and converges to a strictly positive bound under non-Gaussian alternatives.

Results from Monte Carlo simulations on various dimensions and sample sizes indicate that the proposed procedure achieves accurate Type I error control and accurate, generally superior power compared to conventional multivariate tests of normality, particularly at medium and high dimensions.

\end{abstract}

\section{Introduction}
\label{sec:intro}

The Kullback-Leibler (KL) divergence, known also as relative entropy, is a fundamental measure in information theory and statistics used to quantify the difference between two probability distributions. It was introduced by Kullback and
Leibler \cite{Kullback1951} and measures the expected log-likelihood loss (or excess log-loss) that arises when a reference model $g$ is employed instead of the true distribution $f$. Because of this decision-theoretic interpretation, KL divergence occurs naturally in probability-based inference, model selection, goodness-of-fit testing, density comparison, and a wide variety of learning and signal processing tasks.

The KL divergence from $f$ to $g$ is defined as follows for two continuous probability density functions (pdf) $f,g:\mathbb{R}^m \to
\mathbb{R}$:
\begin{equation}\label{eq:kl_divergence}
D_{KL}(f \| g) = \int_{\mathbb{R}^m} f(x) \log \frac{f(x)}{g(x)}\,dx.
\end{equation}
 Since the divergence $D_{KL}(f \| g)\geq 0$ and equals zero only when $f = g$ almost everywhere, it measures the difference between distributions in a “directional” manner despite not being symmetric and thus not being a metric.

The KL divergence is closely related to Shannon entropy, measuring the uncertainty in a single distribution. The differential Shannon entropy $H(f)$ \cite{shannon1948mathematical} for density $f$ is given by
\begin{equation}\label{eq:shannon_entropy}
H(f) = -\int_{\mathbb{R}^{m}} f(x) \log f(x)\,dx,
\end{equation}
and a simple processing yields the following result
\begin{equation}\label{eq:kl_shannon}
D_{KL}(f\|g) = -H(f) - \int_{\mathbb{R}^m} f(x)\log g(x)\,dx.
\end{equation}
Equivalently, $D_{KL}(f\|g)$ is defined as the cross-entropy $-\E_f[\log g(X)]$
minus the entropy $H(f)$, and thus can be considered as the \emph{excess} cross-entropy (or excess expected log loss) that is associated with using the oracle model $g$ in place of $f$.

Equation \eqref{eq:kl_shannon} also clarifies why certain comparison distributions are particularly natural in KL-based procedures.
Under constraints of fixed mean and covariance, the multivariate Gaussian distribution maximizes Shannon entropy; see, e.g., \cite{cover2006elements,Jaynes1957,kullback1997information,jaynes2003probability}. Therefore, among all distributions that share the same first and second moments, the Gaussian is the only distribution that maximizes entropy. Similarly,
$D_{KL}(f\|\phi_{{\mu},{{\Sigma}}})$ is the difference in entropy
$H(\phi_{{\mu},{{\Sigma}}})-H(f)$, and gives a principled measure of deviation from the Gaussian distribution where $\phi_{{\mu},{{\Sigma}}}$ is the adapted Gaussian distribution matching the moments of $f$. The observation motivates Gaussian benchmarks in goodness-of-fit tests and forms the basis for the KL-based procedures developed in this work.

To estimate $D_{KL}(f\|g)$ from data is essential in practical applications including hypothesis testing, density comparison, and anomaly detection. While classical approaches rely on parametric modeling or add-on estimators that are based on histograms and kernel density estimators, these can easily become unstable in high dimensions. The nearest neighbor (NN) methods provide an attractive alternative: they employ the local geometric structure of the sample cloud rather than explicit density reconstruction and are naturally extendable to multivariate settings. For entropy estimation, nearest neighbor ideas date back to Kozachenko and Leonenko \cite{kozachenko1987} and have been extensively studied since \cite{goria2005,leonenko2008,penrose2011,penrose2013,berrett2019a,paninski2003estimation}. As for deviation estimation, kNN-based estimators and their consistency properties have been developed in \cite{perez2008kullback,Wang2009,gao2018,bulinski2019}, among others.

This paper, building upon this line of work, focuses on the nearest neighbor estimators of KL divergence in a multivariate continuous setting and highlights the goodness-of-fit test. The main contributions can be summarized as follows:
\begin{enumerate}[label=\roman*.]
\item We provide an information-theoretic justification for Gauss benchmarks by rephrasing the maximum entropy principle with respect to KL divergence under constraints of mean and covariance.
\item Then, under standard regularity and moment conditions, we review and update the asymptotic properties (consistency, asymptotic unbiasedness, and mean square convergence) of Shannon entropy and KL divergence for nearest neighbor estimators.
\item From these estimators, we generate a KL-based test statistic and examine its finite sample behavior through extensive Monte Carlo experiments, examining the empirical size, power, and the effects of dimension, sample size, and
the nearest neighbor parameter $k$.
\end{enumerate}

The rest of the article has been organized as follows. Section~\ref{sec:maxent} restates the principles of maximum entropy and describes their application to KL divergence and Gaussian benchmarks. Section~\ref{sec:kl_estimation} introduces the nearest neighbor estimators of Shannon entropy and the obtained KL divergence
estimators. Section~\ref{sec:hypothesis_test_KL} provides a framework for hypothesis testing. Section~\ref{sec:numerical_experiments_KL} describes the numerical experiments, and Section~\ref{sec:conclusion} briefly discusses possible extensions before concluding.

\section{Maximum Entropy Principles}
\label{sec:maxent}

The principle of maximum entropy, proposed by Jaynes \cite{Jaynes1957}, indicates that among all probability distributions consistent with a particular set of constraints (usually moment conditions), that distribution with the greatest entropy
is the least informative choice. Because KL divergence and Shannon entropy are closely related, expressions for maximum entropy can often be rewritten in terms of minimum KL divergence relative to a suitable comparison distribution.

First of all, it is useful to formalize the class of acceptable solutions within the standard mean-covariance
setting.

\begin{definition}[Class $\mathcal{K}$]\label{def:classK_KL}
The distribution class $H(f)$ of density functions $f$ on $\mathbb{R}^m$
with finite and moment constraints
\begin{align*}
\int_{\mathbb{R}^m} {x} f({x})\,d{x} &= {\mu}, \\
\int_{\mathbb{R}^m} ({x} - {\mu})({x} - {\mu})^{\top} f({x})\,d{x}
&= {\Sigma},
\end{align*}
where ${ \mu} \in \mathbb{R}^m$ denotes a mean vector and ${{\Sigma}}$ denotes a symmetric positive definite covariance matrix of size $m \times m$.
\end{definition}

In the class $\mathcal{K}$, it is well known that the multivariate Gaussian distribution with mean ${\mu}$ and covariance ${\Sigma}$ maximizes Shannon entropy; see, for example, \cite{cover2006elements,kullback1997information,jaynes2003probability}. Specifically, for any $f \in \mathcal{K}$,
\begin{equation}\label{eq:gaussian_maxent}
H(f) \leq \frac{1}{2}\log\left[(2\pi e)^m \det({\Sigma})\right],
\end{equation} and the equality applies if and only if $f$ is a Gaussian density.
\begin{align}
\phi_{{\mu},{\Sigma}}({x})
= \frac{1}{(2\pi)^{m/2}\det({\Sigma})^{1/2}}
\exp\!\left(-\frac{1}{2}({x}-{\mu})^{\top}{\Sigma}^{-1}({x}-{\mu})\right). \label{eq_gaussian_density}
\end{align}
The following proposition rephrases this classical result of maximum entropy in KL divergence terms and will serve simply as a baseline for our testing procedures.

\begin{proposition}[KL divergence Gauss benchmark]\label{prop:KL_maxent}
Assume that $\mathcal{K}$ is defined as in~{Definition} \ref{def:classK_KL} and
$\phi_{{\mu},{\Sigma}}$ represents a Gaussian density {\eqref{eq_gaussian_density}}, with mean ${\mu}$ and
covariance ${\Sigma}$. Then , for every $f \in \mathcal{K}$,
\begin{equation}\label{eq:KL_gaussian_representation}
D_{KL}\bigl(f \,\|\, \phi_{{\mu},{\Sigma}}\bigr)
= H(\phi_{{\mu},{\Sigma}}) - H(f)
\;\geq\; 0,
\end{equation}
with equality if and only if $f = \phi_{{\mu},{\Sigma}}$ holds almost everywhere (a.e.).
\end{proposition}

\begin{proof}
By using the expression of KL divergence and the expression of Gaussian density,
\[
D_{KL}\bigl(f \,\|\, \phi_{{\mu},{\Sigma}}\bigr)
= \int_{\mathbb{R}^m} f({x}) \log \frac{f({x})}{\phi_{{\mu},{\Sigma}}({x})}\,d{x}
= -H(f) - \int_{\mathbb{R}^m} f({x}) \log \phi_{{\mu},{\Sigma}}({x})\,d{x}.
\]
A direct computation allows for any $f \in \mathcal{K}$,
\[
\int_{\mathbb{R}^m} f({x}) \log \phi_{{\mu},{\Sigma}}({x})\,d{x}
= -\frac{1}{2}\log\bigl((2\pi)^m \det({\Sigma})\bigr) - \frac{m}{2},
\]
since the corresponding first and second moments of $f$ match those of $\phi_{{\mu},{\Sigma}}$.
Therefore
\[
D_{KL}\bigl(f \,\|\, \phi_{{\mu},{\Sigma}}\bigr)
= -H(f) + \frac{1}{2}\log\bigl((2\pi)^m \det({\Sigma})\bigr) + \frac{m}{2}
= H(\phi_{{\mu},{\Sigma}}) - H(f),
\]
\eqref{eq:gaussian_maxent} is a non-negative value. The equality holds only and exclusively under the following condition:
$H(f)=H(\phi_{{\mu},{\Sigma}})$, and therefore $f$ must coincide with
$\phi_{{\mu},{\Sigma}}$ { almost everywhere}.
\end{proof}

Proposition~\ref{prop:KL_maxent} suggests that the principle of maximum entropy can be interpreted as the minimum KL divergence principle in the constrained class $\mathcal{K}$: among all densities {$f \in \mathcal{K}$  with mean ${\mu}$ and covariance
${{\Sigma}}$, there is a unique maximum value of entropy for the Gaussian $\phi_{{\mu},{{\Sigma}}}$ and, equivalently, the KL divergence $D_{KL}(f\|\phi_{{\mu},{{\Sigma}}})$ measures the
\emph{entropy gap} of $f$ relative to this maximum. Specifically, then for any $f \in \mathcal{K}$,
\[
D_{KL}\bigl(f \,\|\, \phi_{{\mu},{\Sigma}}\bigr) = 0
\quad\Longleftrightarrow\quad
f = \phi_{{\mu},{\Sigma}} \ \text{a.e.},
\]
That justifies using Gaussian models as a benchmark for testing goodness-of-fit and model validation.

\begin{remark}
For any density pair $(f,g)$, the Gibbs' inequality states that $D_{KL} \left(f\|g\right) \ge 0$,
while equality applies only if and only if $f = g$ almost everywhere. Proposition~\ref{prop:KL_maxent},
expresses the KL divergence $D_{KL}(f\|\phi_{{\mu},{\Sigma}})$ in terms of a reference distribution $\phi_{{\mu},{\Sigma}}$ that is selected in such a way that the divergence coincides exactly with the entropy difference $H(\phi_{{\mu},{\Sigma}})-H(f)$. 

\end{remark}

A similar table emerges for the multivariate generalized Gaussian (or exponential power) family. For the isotropic case, suppose $f$ is the density of ${{X}}\in\mathbb{R}^m$ supported on $\mathbb{R}^m$ and assume that for some $s>0$ the moment $\E\|{{X}}\|^{s}$ is finite. Results such as \cite{kapur1989,solaro2004,berrett2019a} yield entropy inequalities of the following form
\[
H(f) \leq m \log\!\bigl(\alpha(m,s)\,\E\|{{X}}\|^{s}\bigr),
\]
holds for an explicit constant $\alpha(m,s)$ associated with $m$ and $s$, if and only if ${{X}}$ satisfies the isotropic generalized Gaussian distribution $GG_{\tau}(m,s)$ for the corresponding scale parameter $\tau$. The same statement holds for multivariate exponential power distributions with general location and variance matrices; see, e.g., \cite{solaro2004,berrett2019a}. For these classes, the corresponding constraint is formulated in terms of radial moments instead of second-order moments. As in Proposition~\ref{prop:KL_maxent}, these entropy bounds can be reformulated as KL inequalities associated with the corresponding generalized Gaussian benchmark $g_{GG}$. For the case where $f$ and $g_{GG}$ share the same constraint (e.g., the same $\E\|{{X}}\|^{s}$ in an isotropic medium), we have the following result
\[
D_{KL}\bigl(f \,\|\, g_{GG}\bigr)
= H(g_{GG}) - H(f) \;\geq\; 0,
\]
and the equality is valid if and only if $f = g_{GG}$ almost everywhere. From this point of view, KL-based tests can also be constructed with non-Gaussian maximum entropy benchmarks; in this paper, however, we restrict our attention to the classical Gaussian mean-covariance case.

\section{Nearest-Neighbor Estimators of the Shannon Entropy and KL divergrence}
\label{sec:kl_estimation}

To estimate Shannon entropy and the KL divergence is a crucial step in many nonparametric inference problems. For multivariate settings, the nearest neighbor methods provide a practical alternative to histogram and kernel-based estimators, in particular when dimension $m$ is moderate and explicit multivariate density estimation has become unstable.

We recall the $k$-nearest neighbor (kNN) estimator of Shannon entropy, describe how it leads to the KL divergence estimator, and outline the asymptotic properties that form the basis of our testing procedure.

\subsection{Nearest-neighbor estimation of Shannon entropy}

Suppose {${X}_1,\dots,{X}_N$} are i.i.d. random vectors in $\R^m$ with density $f$, and $Y_1,\dots,Y_M$ are i.i.d. random vectors in $\R^m$ with density $g$, which are independent of the $X$ sample. Let us assume that $f$ and $g$ are (Lebesgue) densities which are almost everywhere continuous and almost everywhere positive on their supports, and satisfy the regularity and tail/moment conditions specified in Section~\ref{sec:kl_estimation} when specifying asymptotic results.

The nearest neighbor estimator of Shannon entropy was introduced by Kozachenko and Leonenko \cite{kozachenko1987}, reviewed in more detail in \cite{goria2005,leonenko2008}, and it is based on a geometric approach to local density: The sample $X_i$ becomes sparse around $X_i$ when its $k$th nearest neighbor is distant, and the principal density decreases; when neighbors are close, the principal density increases.

We assume that $\rho_{i,k,N}$ denotes the Euclidean distance from $X_i$ to its $k$th nearest neighbor among $\{X_j\}_{j\neq i}$. The standard version of the Kozachenko–Leonenko (kNN) estimator of Shannon entropy is the following
\begin{equation}\label{eq:kNN_entropy_estimator}
\widehat{H}_{N,k}(f)
= \psi(N) - \psi(k) + \log V_m + \frac{m}{N}\sum_{i=1}^{N}\log \rho_{i,k,N},
\end{equation}
where $V_m = \pi^{m/2}/\Gamma(\frac{m}{2}+1)$ is the volume of the unit ball in $\mathbb{R}^m$ and $\psi(\cdot)$ is the digamma function.

The estimator can be seen as an extension of the statement below:
\[
H(f) = -\E[\log f(X_1)],
\]
where we replace $f(X_i)$ with
the local kNN density approximation guided by $\rho_{i,k,N}$.

Adjustment parameter $k$ balances bias and variance: a small $k$ results in lower bias but higher variance, whereas a larger $k$ increases bias to stabilize the estimate. While in practice, $k$ is typically fixed to an integer (e.g., $k\in\{3,5,10\}$, consistent with $k$ constant asymptotic theory and sufficient for the dimensions considered in our experiments.

\begin{remark}
Some equivalent presentations include the constants in a single logarithm, for example
\[
\widehat{H}_{N,k}(f)
= \frac{1}{N}\sum_{i=1}^{N}\log\!\left(\rho_{i,k,N}^m V_m e^{\psi(N)-\psi(k)}\right),
\]
and uses $\psi(N)\approx \log(N-1)$ for large $N$. As it is standard in the theoretical literature and makes the finite sample constants explicit, we preserve the digamma form in
\eqref{eq:kNN_entropy_estimator}.
\end{remark}

\subsection{Nearest-neighbor estimation of KL divergence}

To estimate $D_{KL}(f\|g)$, a direct additive approach estimates $f$ and $g$ separately and then combines these estimates using \eqref{eq:kl_divergence}. In moderate or high dimensions, this strategy often inherits the instability of multivariate density estimation.
Let $X_1,\dots,X_N\sim f$ and $Y_1,\dots,Y_M\sim g$ be independent samples, and let $\nu_{i,k}$ denotes the  Euclidean distance from $X_i$ to its $k$th nearest neighbor in the $Y$ sample. According to \cite{perez2008kullback,leonenko2008,Wang2009}, a commonly used kNN estimator for $D_{KL}(f\|g)$ is
\begin{equation}\label{eq:KL_kNN_estimator}
\widehat{D}_{KL}(f\|g)
= \frac{m}{N}\sum_{i=1}^{N}\log\frac{\nu_{i,k}}{\rho_{i,k,N}}
+ \psi(M) - \psi(N-1).
\end{equation}
Statistics compare the typical neighbor sizes around each $X_i$ when neighbors are obtained from the same distribution (via $\rho_{i,k,N}$) and from the second distribution (via $\nu_{i,k}$). When $f$ and $g$ are closely related, these radii are comparable and the estimated deviation is small; systematic differences produce a positive estimate.

Intuitively, \eqref{eq:KL_kNN_estimator} follows from the following identity
\[
D_{KL}(f\|g) = \E[\log f(X_1)] - \E[\log g(X_1)],
\]
 where each log-density is substituted with a kNN-based approximation constructed from the associated sample. In the difference, the volume term $V_m$ cancels out, which explains its absence in \eqref{eq:KL_kNN_estimator}. For large $M,N$ values, $\psi(M)-\psi(N-1)\approx \log\!\bigl(M/(N-1)\bigr)$ can be employed, but we retain the digamma form to follow the standard bias correction.

\subsection{Asymptotic properties}
There is a well-developed asymptotic theory for kNN entropy and divergence estimators;
see, among others,
\cite{kozachenko1987,goria2005,leonenko2008,penrose2011,penrose2013,Wang2009,berrett2019a,cadirci2022entropy}.
For our objectives, it suffices to provide a representative consistency conclusion for
a fixed $k$ under standard smoothness and tail conditions.

\begin{theorem}[Consistency of kNN estimators]\label{thm:consistency_nn}
Suppose that $f$ and $g$ are densities on $\mathbb{R}^m$ and that $H(f)$ and $H(g)$ are finite. Assume that $f$ and $g$ are almost everywhere continuous, upper bounded, and satisfy the following moment/tail condition
We assume that they satisfy a moment/tail condition of the form $\E\|X\|^{m+\delta}<\infty$ and $\E\|Y\|^{m+\delta}<\infty$, where $X\sim f$ and $Y\sim g$. Let $k\ge 1$ be a constant. Then, as $N,M\to\infty$,
\[
\widehat{H}_{N,k}(f)\xrightarrow{\text{a.s.}} H(f),
\qquad
\widehat{D}_{KL}(f\|g)\xrightarrow{\text{a.s.}} D_{KL}(f\|g).
\]
Under the same conditions, the estimators will also be consistent in terms of the mean square error, i.e., $\Var(\widehat{H}_{N,k}(f))\to 0$ and $\Var(\widehat{D}_{KL}(f\|g))\to 0$, and will be asymptotically unbiased.
\end{theorem}
The above statement is consistent with the $L^2$-consistency framework introduced in
\cite{penrose2011,penrose2013}. In particular, Cadirci et al.
\cite{cadirci2022entropy} employ Poissonization techniques to estimate the mean square convergence of the Kozachenko-Leonenko entropy estimator for a random fixed $k\ge 1$ in the context of entropy-based testing.
A similar set of techniques, together with the unbiased form
\eqref{eq:KL_kNN_estimator}, supports the associated mean square consistency claims for the kNN KL estimator used here.

\begin{remark}
Depending on whether the support is bounded or unbounded, different sources specify slightly different sufficiency conditions. In the case of bounded support, it is generally assumed that there is a positive lower bound on the density within the support; in the case of unbounded
support (including Gauss and related families), conditions on moments and tails substitute for global lower bounds. In our numerical experiments, we have chosen the conditions in Theorem~\ref{thm:consistency_nn} to cover the distributions used, and they remain close to the standard assumptions in the cited references.
\end{remark}
In summary, the nearest neighbor methods offer a simple and theoretically grounded way to estimate both Shannon entropy and KL divergence in multivariate problems.
These methods avoid explicit multivariate density reconstruction,
are computationally efficient, and form the basis for the KL-based goodness-of-fit procedures
discussed in the rest of the paper.

\section{Goodness-of-fit tests based on KL divergence}
\label{sec:hypothesis_test_KL}

Based on the  KL divergence, we perform a goodness-of-fit test for multivariate normality. Let {${X}_1,\dots,{X}_N$} be i.i.d. random vectors in $\R^m$ with density $f$, finite mean ${\mu} = \E[{{X}}]$, and covariance
${\Sigma} = \Var({{X}})$, are i.i.d random vectors in $\R^m$. We test the composite null hypothesis
\[
H_0:\ f \in \mathcal{F}_{\mathcal N}
\quad\text{(family of $m$-variable Gaussians with unconditional mean and covariance)}
\]
against the general alternative $H_1:\ f\notin \mathcal{F}_{\mathcal N}$.

Let
\[
\bar{ X}_N = \frac1N\sum_{i=1}^N X_i,
\qquad
{{S}}_N = \frac{1}{N-1}\sum_{i=1}^N (X_i-\bar{ X}_N)(X_i-\bar{ X}_N)^\top
\]
be the sample mean and sample covariance matrix, respectively, and suppose that $\phi_{\bar{ X}_N,{{S}}_N}$ is the Gaussian density having this mean and covariance.
We recall from Proposition~\ref{prop:KL_maxent} that, within the class
$\mathcal{K}$ of distributions with fixed mean ${\mu}$ and covariance ${{\Sigma}}$,
the multivariate Gaussian density $\phi_{{\mu},{{\Sigma}}}$ uniquely maximizes the
Shannon entropy and satisfies
\[
D_{KL}(f\|\phi_{{\mu},{{\Sigma}}}) = H(\phi_{{\mu},{{\Sigma}}}) - H(f) \ge 0,
\]
with equality if and only if $f=\phi_{{\mu},{{\Sigma}}}$ almost everywhere and
\[
H(\phi_{{\mu},{{\Sigma}}}) = \frac12 \log\bigl[(2\pi e)^m \det({{\Sigma}})\bigr].
\]
Based on this idea, let's define the KL-based test statistic
\begin{equation}
\label{eq:KL_test_stat}
T_{N,k}^{\mathrm{KL}}
  := H\!\bigl(\phi_{\bar{ X}_N,{{S}}_N}\bigr)
     - \widehat H_{N,k}(f)
  = \frac12 \log\bigl[(2\pi e)^m \det({{S}}_N)\bigr]
    - \widehat H_{N,k}(f),
\end{equation}
where $\widehat H_{N,k}(f)$ is an estimator of the Shannon entropy given in~\eqref{eq:kNN_entropy_estimator} using the $k$th nearest neighbor.
As constructed, $T_{N,k}^{\mathrm{KL}}$ gives an estimate of the KL divergence
$D_{KL}(f\|\phi_{{\mu},{{\Sigma}}})$ using the identity $D_{KL}(f\|g)=H(g)-H(f)$.

\begin{theorem}
\label{thm:KL_test_consistency}
Suppose that $f$ is a continuous density on $\R^m$ with finite second
moments and that the conditions of Theorem~\ref{thm:consistency_nn}
are satisfied. For any fixed $k\ge1$,
\[
T_{N,k}^{\mathrm{KL}}
\xrightarrow{p}
H(\phi_{{\mu},{{\Sigma}}}) - H(f)
= D_{KL}(f\|\phi_{{\mu},{{\Sigma}}})
\quad\text{as } N\to\infty.
\]
In particularly,
\[
T_{N,k}^{\mathrm{KL}} \xrightarrow{p} 0
\quad\text{if } X\sim N_m({\mu},{{\Sigma}}),
\]
and $T_{N,k}^{\mathrm{KL}} \xrightarrow{p} \xi(f)>0$ for any non-Gaussian
$f\in \mathcal{K}$ having the same mean ${\mu}$ and covariance ${{\Sigma}}$.
\end{theorem}

\noindent \emph{Sketch of proof.} We decompose
\[
T_{N,k}^{\mathrm{KL}}
= \underbrace{\bigl(H(\phi_{\bar{{X}}_N,{{S}}_N}) - H(\phi_{{\mu},{\Sigma}})\bigr)}_{\text{(I)}}
+ \underbrace{\bigl(H(\phi_{{\mu},{\Sigma}}) - H(f)\bigr)}_{\text{(II)}}
+ \underbrace{\bigl(H(f) - \widehat{H}_{N,k}(f)\bigr)}_{\text{(III)}}.
\]
According to the Law of Large Numbers, ${{S}}_N \to {\Sigma}$ almost surely, hence (I) $\to 0$,
since $N\to\infty$ due to the continuous nature of $\log\det(\cdot)$ on the set of positive definite matrices.
Under Proposition~\ref{prop:KL_maxent}, term (II) equals $D_{KL}(f\|\phi_{{\mu},{\Sigma}})$
and vanishes to zero in probability under the assumptions of Theorem~\ref{thm:consistency_nn}.
\hfill $\square$

\medskip

Practically, it is not possible to obtain the zero distribution of $T_{N,k}^{\mathrm{KL}}$ in closed form. We therefore use parametric bootstrapping under an adapted Gaussian model to calibrate the test. Under $N_m(\bar{ X}_N,{{S}}_N)$, we perform calibration using parametric bootstrap. We simulate bootstrap samples for the selected neighborhood size $k$ and significance level $\alpha\in(0,1)$
$X_{1:N}^{*(b)} \sim N_m(\bar{ X}_N,{{S}}_N)$, calculate $T_{N,k}^{\mathrm{KL}}$ for each sample, and take the $(1-\alpha)$-quantile of the resulted bootstrap values as the critically threshold $t_\alpha$. When $T_{N,k}^{\mathrm{KL,obs}} \ge t_\alpha$, we reject $H_0$.

\begin{remark}[KL divergence as expected log-loss]
\label{rem:KL_risk}
The divergence
\[
D_{KL}(f\|\phi_{{\mu},{{\Sigma}}})
  = \int_{\R^m} f(x)\log\!\frac{f(x)}{\phi_{{\mu},{{\Sigma}}}(x)}\,\mathrm{d}x
\]
is interpreted as the average log-likelihood loss incurred by employing the Gaussian model $\phi_{{\mu},{{\Sigma}}}$ to substitute the true density $f$. In fact,
\[
D_{KL}(f\|\phi_{{\mu},{{\Sigma}}})
  = -\,\E_f[\log \phi_{{\mu},{{\Sigma}}}(X)] + H(f),
\]
which means it measures the excess cross-entropy (equivalently, it's the increase in expected log-loss) relative to an oracle model that knows $f$.
\end{remark}
Since analytical null distributions for $T_{N,k}^{\mathrm{KL}}$ are difficult to obtain,
so we calibrate critical values using a parametric bootstrap under $H_0$
(Algorithm~\ref{alg:KL_bootstrap}), by following related entropy- and KL-based
goodness-of-fit tests in
\cite{Song2002,GirardinLequesne2017,LequesneRegnault2020,Noughabi2013,Noughabi2019}.

\begin{algorithm}
\caption{Bootstrap calibration for $T_{N,k}^{\mathrm{KL}}$}
\label{alg:KL_bootstrap}
\begin{algorithmic}[1]
\State \textbf{Input:} data $X_{1:N}$, number of bootstrap replications $B$, significance level $\alpha$, neighborhood size $k$.
\State Calcuate $\bar{ X}_N$ and ${{S}}_N$.
\State Calculate $T_{N,k}^{\mathrm{KL,obs}}$ from $X_{1:N}$ by using \eqref{eq:KL_test_stat}.
\For{$b = 1, \dots, B$}
  \State Simulate $X^{(b)}_{1:N} \sim N_m(\bar{ X}_N,{{S}}_N)$.
  \State Calculate $T_{N,k}^{(b)} = T_{N,k}^{\mathrm{KL}}\!\bigl(X^{(b)}_{1:N}\bigr)$.
\EndFor
\State Suppose $t_\alpha$ is the $(1-\alpha)$-quantile of $\{T_{N,k}^{(b)}\}_{b=1}^B$.
\State \textbf{Decision:} reject $H_0$ if $T_{N,k}^{\mathrm{KL,obs}} \ge t_\alpha$.
\end{algorithmic}
\end{algorithm}

\section{Numerical Experiments}
\label{sec:numerical_experiments_KL}

In this section, numerical experiments are reported for the KL-based statistic $T_{N,k}^{\mathrm{KL}}$ specified in \eqref{eq:KL_test_stat}. There are three objectives. We first display the finite sample behavior of the statistic in terms of sample size $N$, dimension $m$, and neighborhood size $k$. Secondly, the empirical power is examined against structured non-Gaussian alternatives, including light and heavy tailed deviations. Finally, in practical applications, we provide calibrated 5\% critical values. 
As noted, Monte Carlo results are based on independent repetitions, and error bars are based on an empirical standard deviation. For all experiments, we use Euclidean distances in the $k$-nearest neighbor calculations.

\subsection{Monte Carlo study of $T_{N,k}^{\mathrm{KL}}$: convergence and stability}
\label{subsec:mc_convergence}

We examine the convergence behavior of $T_{N,k}^{\mathrm{KL}}$ first. Propsition~\ref{prop:KL_maxent} implies that, under multivariate normality, the population target is $D_{KL}(f\|\phi_{{\mu},{{\Sigma}}})=0$ and therefore the statistic should converge to zero as $N$ increases. To evaluate the robustness beyond the Gaussian case, we also simulate examples from generalized Gaussian models having shape parameter $s$, noting that $s=2$ corresponds to the Gaussian reference value in this family.
We performed $M=100$ independent repetitions for each $s \in \{1,2,4,8\}$ and $m \in \{1,2,3\}$. Figure~\ref{fig:KL_consistency_k_1} reports the results for
$N \in \{1000,\dots,5000\}$. In the Gaussian case ($s=2$), the statistic approaches zero. For $s\neq 2$, the statistic stabilizes away from zero, consistent with convergence to a positive KL bound under fixed non-Gaussian alternatives.

\begin{figure}[H]
\centering
\includegraphics[width=\textwidth]{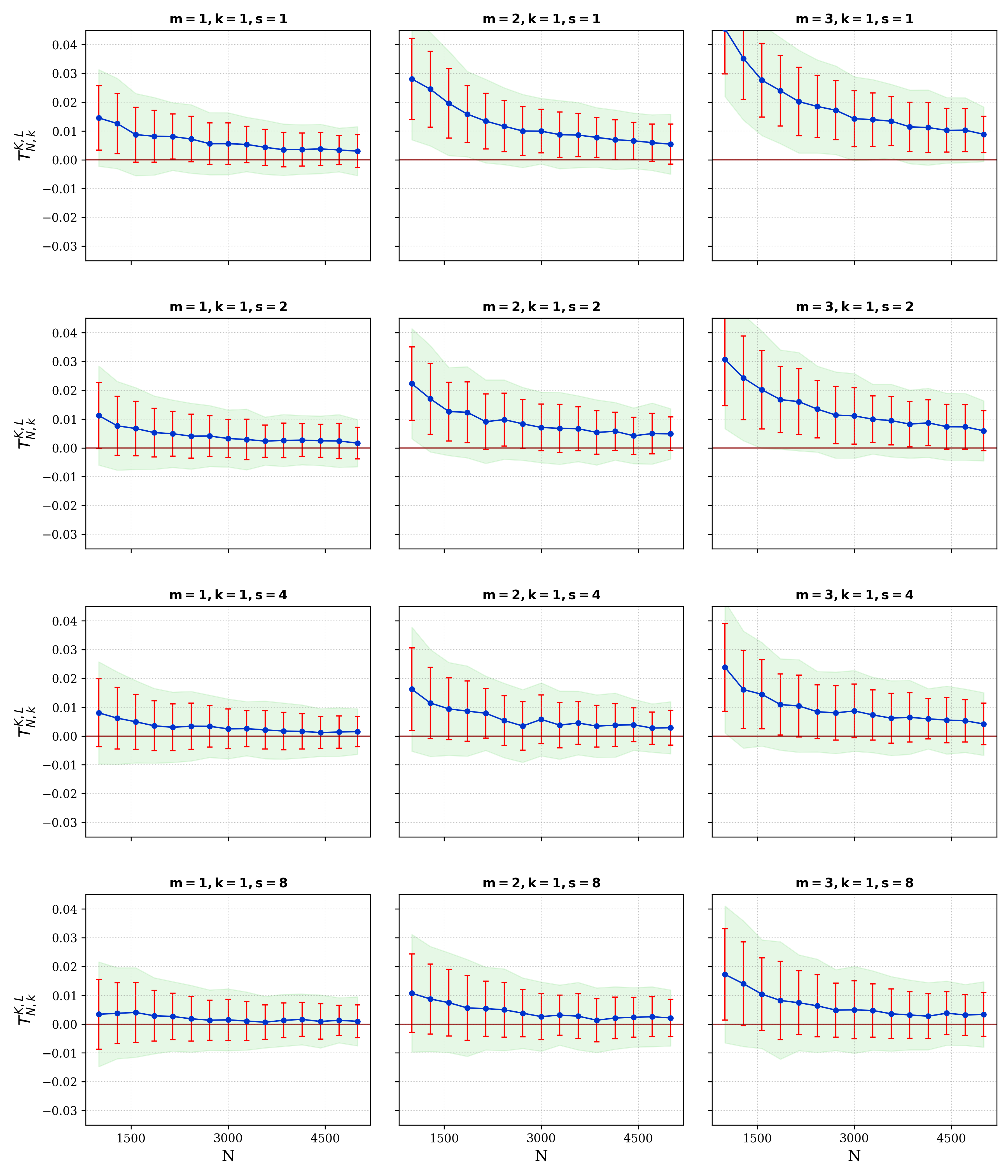}\vspace{-3pt}
\caption{Convergence behavior of $T_{N,k}^{\mathrm{KL}}$ for the neighborhood dimension $k=1$. Error bars represent one standard deviation empirically over $M=100$ iterations. This corresponds to the Gaussian benchmark, $s=2$; here, the statistic becomes concentrated near zero as $N$ increases. The convergence behavior of $T_{N,k}^{\mathrm{KL}}$ for $k=1$.}
\label{fig:KL_consistency_k_1}
\end{figure}

Subsequently, we repeat the experiment for smaller sample sizes $N \in \{100,\dots,1000\}$ and
$k \in \{1,2,3\}$ to investigate the finite sample regime and the effect of $k$.
Figure~\ref{fig:kl_TNK_KL_grid_shrinking_bars} illustrates that increasing $k$ reduces variability essentially without changing the central tendency. These results are consistent with the standard bias-variance tradeoff in kNN functionals: larger neighborhoods provide an additional mean and therefore provide a smaller variance, but this usually comes at the cost of a modest increase in bias \cite{berrett2019a}. In the current $(m,N)$ range, variance decrease is the dominant and visible effect.

\begin{figure}[H]
\centering
\includegraphics[width=\textwidth]{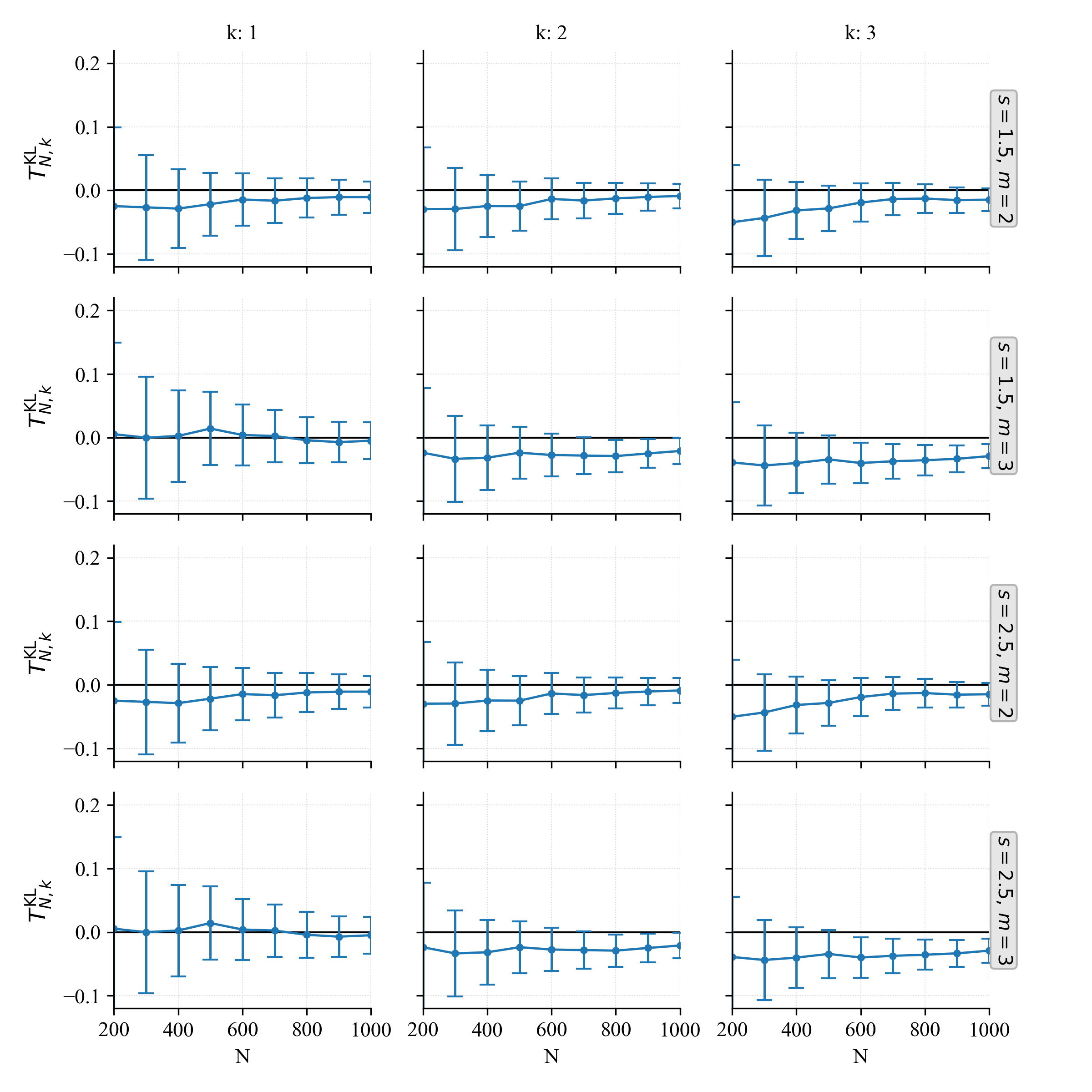}\vspace{-3pt}
\caption{Finite sample stability of $T_{N,k}^{\mathrm{KL}}$ over neighborhood dimensions $k\in\{1,2,3\}$. Error bars indicate one empirical standard deviation. Variability decreases with $N$ and further decreases with larger neighborhoods.}
\label{fig:kl_TNK_KL_grid_shrinking_bars}
\end{figure}

\subsection{Power of the goodness-of-fit test}
\label{subsec:power_analysis}

For each configuration, we compute the bootstrap critical value utilizing Algorithm~\ref{alg:KL_bootstrap} using the fitted Gaussian model $N_m(\bar{ X}_N,{{S}}_N)$. We evaluate empirical power for testing multivariate normality at the significance level $\alpha=0.05$. In this case, the empirical power is the proportion of Monte Carlo replications (generated under the alternative hypothesis) for which the observed
$T_{N,k}^{\mathrm{KL}}$ exceeds the bootstrap critical value.

In power experiments, $M=1000$ replications and sample sizes $N\in\{500,1000\}$ are used. We consider two complementary classes of alternatives: (i) deviations within a generalized Gaussian family (controlled tail behavior),
and (ii) heavy-tailed Student-type alternatives (departures outside the generalized Gaussian family for finite degrees of freedom).

\subsubsection{Power against generalised Gaussian alternatives}

For the first experiment, we generalize the Gaussian family by shifting the shape parameter away from the Gaussian reference value. Specifically,
we keep the location and variance fixed and compare the Gaussian case with $X\sim GG(m,v)$ where $v\neq 2$. Figure~\ref{fig:KL_power_detailed}
displays the resulting power curves.

There is consistency between the two model sizes. First, power increases as the deviation from the Gaussian distribution increases, reflecting the consistency of $T_{N,k}^{\mathrm{KL}}$ against fixed alternatives. Second, increasing the sample size from $N=500$ to $N=1000$ produces a noticeable increase in sensitivity. Choice of neighborhood size $k$ has a secondary but systematic effect: larger neighborhoods often produce smoother and less variable power curves, consistent with the variance reduction often observed in Figure~\ref{fig:kl_TNK_KL_grid_shrinking_bars}.

\begin{figure}[H]
\centering
\includegraphics[width=\textwidth]{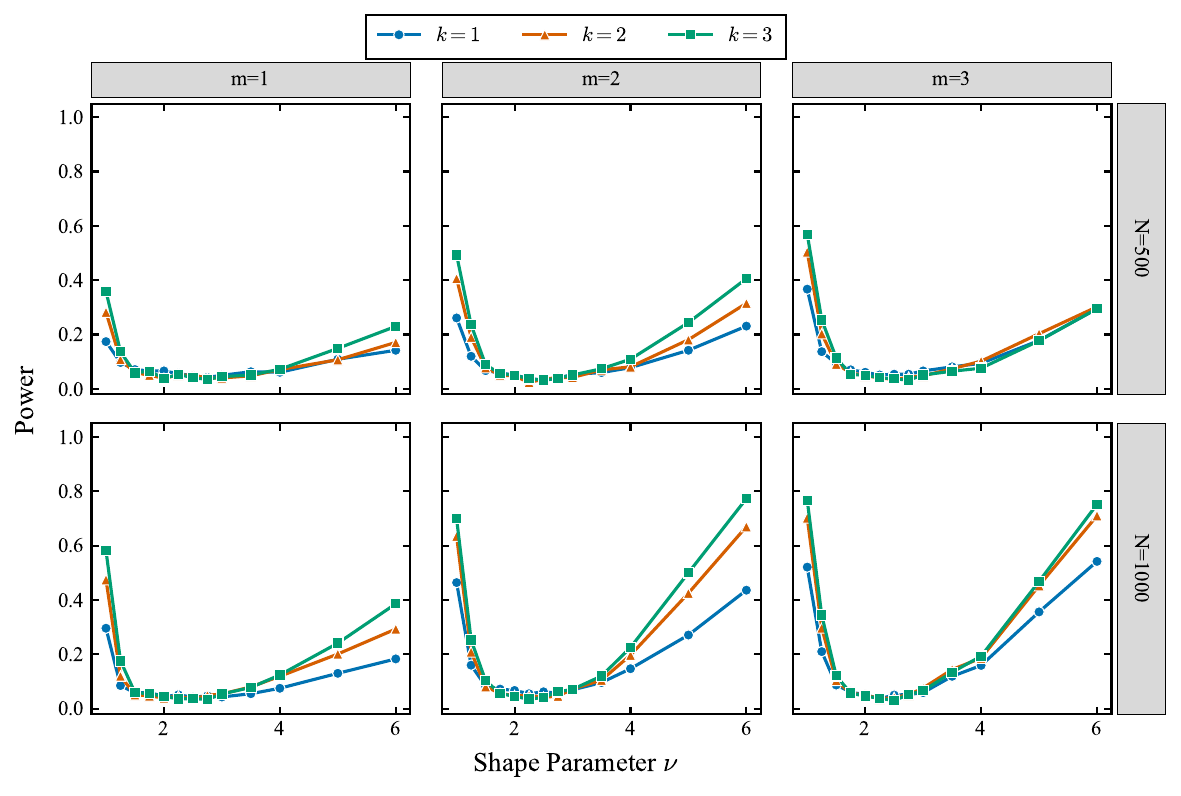}\vspace{-3pt}
\caption{Curves compare $k\in\{1,2,3\}$ across dimensions; rows correspond to $N=500$ and $N=1000$, with $M=1000$ replications. Curves compare $k\in\{1,2,3\}$ across dimensions; rows correspond to $N=500$ and $N=1000$,
with $M=1000$ replications.}
\label{fig:KL_power_detailed}
\end{figure}

\subsubsection{Power against Student-type alternatives}

Next, we will consider heavy-tailed Student's t-type alternatives parameterized by degrees of freedom $\nu$. Here, smaller $\nu$ values correspond to heavier tails. Figure~\ref{fig:power_curves} illustrates that the power approximates one in heavy-tailed cases (small $\nu$) and decreases as $\nu$ increases. It is expected that this trend occurs because larger $\nu$ values produce distributions that are increasingly similar to the Gaussian reference, making it harder to distinguish the alternative. For $N=500$ and $N=1000$, the distinction between the curves once again confirms that larger sample sizes increase sensitivity. For most configurations, moderate neighborhood sizes produce slightly more stable power profiles, which is consistent with the stability effects observed in convergence experiments.

\begin{figure}[H]
\centering
\includegraphics[width=\textwidth]{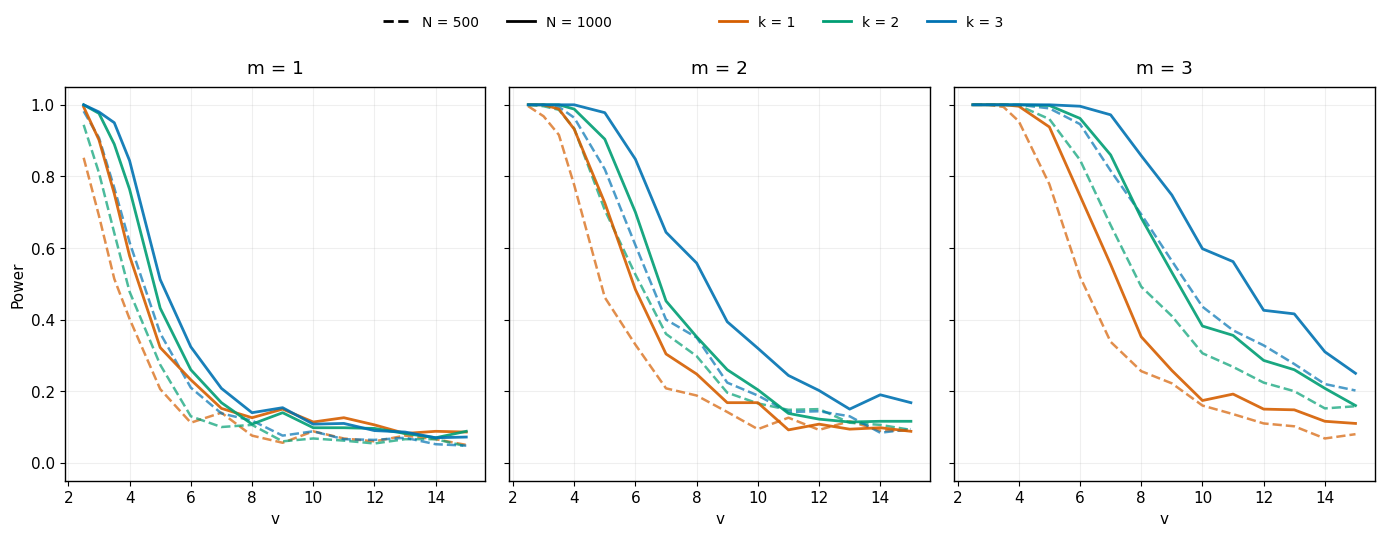}\vspace{-3pt}
\caption{The empirical power of $T_{N,k}^{\mathrm{KL}}$ compared to studentized alternatives. The solid and dashed lines represent $N=1000$ and $N=500$, respectively; the lines compare $k\in\{1,2,3\}$.}
\label{fig:power_curves}
\end{figure}

\subsection{Empirical distribution of the standardised statistic}
\label{subsec:empirical_distribution}
Although calibrated using bootstrap testing, examining the shape of the statistic after standardization can be informative.
\[
Z_{N,k}=\frac{T_{N,k}^{\mathrm{KL}}-\widehat{\mu}_{N,k}}{\widehat{\sigma}_{N,k}},
\]
where $\widehat{\mu}_{N,k}$ and $\widehat{\sigma}_{N,k}$ represent the empirical mean and standard deviation computed from Monte Carlo replications using the Gaussian benchmark. Figure~\ref{fig:kl_standardised_kde_m2N1000} contrasts kernel density estimates of $Z_{N,k}$ at $N=1000$ and $k\in\{1,2,3\}$ with the standard normal density. In this regime, the fit is close, indicating that the normal approximation may be appropriate for descriptive purposes.
To emphasize that this evidence is exploratory: throughout the paper, rejection thresholds were computed using parametric bootstrap within Algorithm~\ref{alg:KL_bootstrap}, which is not based on the assumed asymptotic normal limit.

\begin{figure}[H]
\centering
\includegraphics[width=\textwidth]{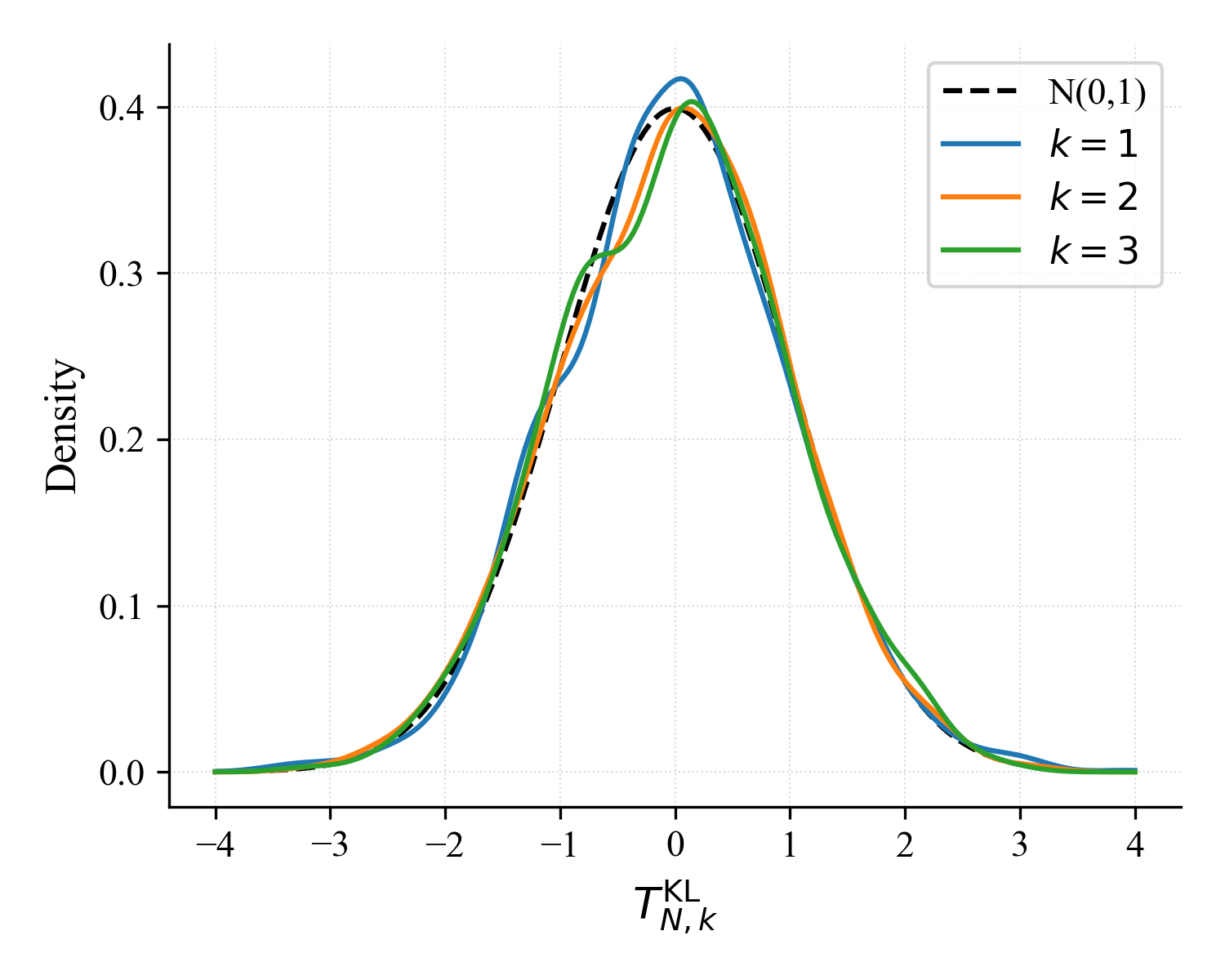}\vspace{-3pt}
\caption{For $N=1000$ and $k\in\{1,2,3\}$, kernel density estimates of the standardized statistic $Z_{N,k}$ were evaluated and compared with the standard Gaussian  density $\mathcal{N}(0,1)$.}
\label{fig:kl_standardised_kde_m2N1000}
\end{figure}
Figure~\ref{fig:kl_qqplots_m2_N1000} gives a complementary Q-Q plot.
Consistent with the KDE evidence, the empirical quantiles closely approximate the reference line, including the tails.

\begin{figure}[H]
\centering
\includegraphics[width=\textwidth]{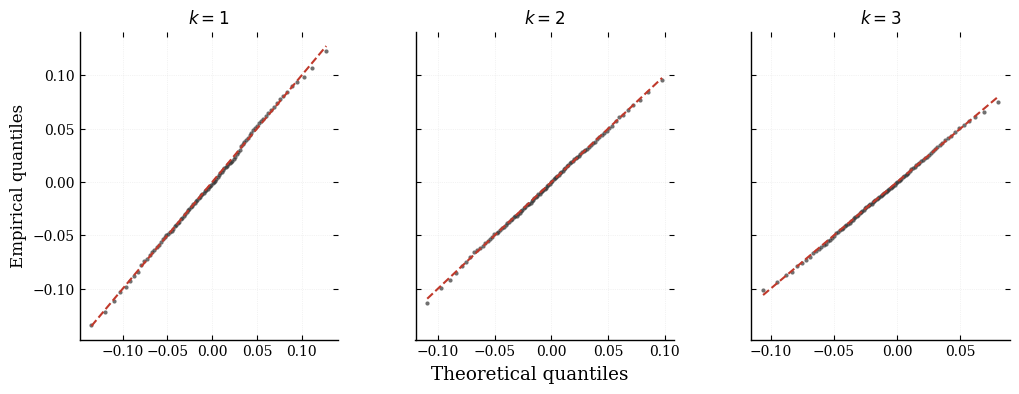}\vspace{-3pt}
\caption{Normally distributed Q-Q plots of the standardized statistic $Z_{N,k}$ for $N=1000$ and $k\in\{1,2,3\}$.}
\label{fig:kl_qqplots_m2_N1000}
\end{figure}

\subsection{Finite-sample rate of convergence: a log-log regression diagnostic}
\label{subsec:convergence_rate}

To summarize the convergence rate of $T_{N,k}^{\mathrm{KL}}$ under the Gaussian  benchmark, we fit a simple log-log regression to the Monte Carlo estimates of
$\bigl|\E[T_{N,k}^{\mathrm{KL}}(m,s)]\bigr|$ across increasing sample sizes:
\begin{equation}
\label{eq:regression_model}
\log \Bigl|\mathbb{E}\bigl[T_{N,k}^{\mathrm{KL}}(m,s)\bigr]\Bigr|
= \alpha_{m,s,k} + \beta_{m,s,k}\,\log N.
\end{equation}
For kNN entropy functionals, the leading bias term is typically of order $\mathcal{O}(N^{-1/2})$ in normal regimes. Based on this intuitive approach, in the Gaussian benchmark (corresponding to $s=2$ in the generalized Gaussian family), $\beta_{m,s,k}\approx -\tfrac12$ is expected. Thus, deviations from this benchmark can be explained as slower decay or stronger finite sample effects.

\begin{table*}[t]
\centering
\small
\renewcommand\arraystretch{1.3}
\caption{Estimated slope coefficients $\beta$ from the log-log regression}
\eqref{eq:regression_model} for $\bigl|\E[T_{N,k}^{\mathrm{KL}}(m,s)]\bigr|$ against $N$.
\label{table:kl-convergence-slopes}
\begin{tabular*}{\textwidth}{@{\extracolsep{\fill}}c|ccc|ccc|ccc}
\toprule
\multirow{2}{*}{$s$} & \multicolumn{3}{c|}{$m = 1$} & \multicolumn{3}{c|}{$m = 2$} & \multicolumn{3}{c}{$m = 3$} \\
\cline{2-10}
 & $k = 1$ & $k = 2$ & $k = 3$ & $k = 1$ & $k = 2$ & $k = 3$ & $k = 1$ & $k = 2$ & $k = 3$ \\
\midrule
1.0 & 0.0092 & 0.0085 & 0.0078 & 0.0124 & 0.0110 & 0.0095 & 0.0145 & 0.0131 & 0.0118 \\
1.5 & 0.0041 & 0.0035 & 0.0029 & 0.0058 & 0.0049 & 0.0042 & 0.0072 & 0.0064 & 0.0056 \\
2.0 & 0.0004 & 0.0002 & 0.0001 & 0.0008 & 0.0005 & 0.0003 & 0.0012 & 0.0009 & 0.0007 \\
2.5 & 0.0006 & 0.0004 & 0.0002 & 0.0011 & 0.0009 & 0.0007 & 0.0018 & 0.0015 & 0.0012 \\
3.0 & 0.0015 & 0.0012 & 0.0009 & 0.0024 & 0.0020 & 0.0017 & 0.0031 & 0.0028 & 0.0024 \\
4.0 & 0.0032 & 0.0028 & 0.0023 & 0.0045 & 0.0039 & 0.0034 & 0.0056 & 0.0051 & 0.0046 \\
\bottomrule
\end{tabular*}
\end{table*}

Table~\ref{table:kl-convergence-slopes} shows estimations of the slopes for dimensions $m$, neighborhood sizes $k$, and shape parameters $s$. These values reach their minimum near the Gaussian reference value $s=2$, in which case convergence is fastest. The slopes increase slightly but remain small as we move away from $s=2$, which is consistent with the stable behavior of the statistic under both heavy and light tailed generators. Furthermore, as $k$ increases from $1$ to $3$, we notice a slight decrease in $\beta$, reflecting improved stability in the nearest neighbor estimation for larger neighborhoods..

\subsection{Critical values for practical implementation}
\label{subsec:critical_values}
We estimate 5\% critical values $t_{0.05}$ by parametric bootstrap under the fitted Gaussian model
(Algorithm~\ref{alg:KL_bootstrap}), using $M=1000$ replications. Table~\ref{table:crt_val_KL} shows these thresholds for sample sizes $N\in\{100,\dots,1000\}$, dimensions $m\in\{2,3\}$, and neighborhood sizes
$k\in\{1,2,3\}$. (Rows indexed by $s$ can be understood as a sensitivity check across different generators; for Gaussian calibration, the relevant row is $s=2$.)

These critical values slowly decrease as $N$ increases, reflecting that $T_{N,k}^{\mathrm{KL}}$ becomes close to zero under the Gaussian benchmark. With a fixed $N$, thresholds are larger for $m=3$ than for $m=2$, consistent with the increased challenges of local entropy estimation in higher dimensions. Typically, increasing $k$ yields slightly smaller thresholds, aligning with the variance reduction illustrated in Figure~\ref{fig:kl_TNK_KL_grid_shrinking_bars}.

\begin{table}[htbp]
\centering
\small
\renewcommand\arraystretch{1.25}
\caption{Statistics for KL-based tests $T_{N,k}^{\mathrm{KL}}(m,s)$ critical values $t_{0.05}$ at the 5\% level were estimated using $M=1000$ repeated parametric bootstrap. For Gaussian goodness-of-fit, row $s=2$ corresponds to the Gaussian reference value.}
\label{table:crt_val_KL}
\begin{tabular*}{\textwidth}{@{\extracolsep{\fill}} cc | ccc | ccc }
\hline
\multirow{2}{*}{$s$} & \multirow{2}{*}{$N$} & \multicolumn{3}{c|}{$m=2$} & \multicolumn{3}{c}{$m=3$} \\
\cline{3-8}
 & & $k=1$ & $k=2$ & $k=3$ & $k=1$ & $k=2$ & $k=3$ \\
\hline
1.5 & 100  & 0.1428 & 0.1315 & 0.1254 & 0.1652 & 0.1540 & 0.1488 \\
    & 200  & 0.0985 & 0.0912 & 0.0865 & 0.1187 & 0.1095 & 0.1052 \\
    & 300  & 0.0789 & 0.0724 & 0.0691 & 0.0954 & 0.0881 & 0.0843 \\
    & 400  & 0.0672 & 0.0615 & 0.0588 & 0.0815 & 0.0748 & 0.0712 \\
    & 500  & 0.0594 & 0.0541 & 0.0519 & 0.0721 & 0.0665 & 0.0634 \\
    & 600  & 0.0538 & 0.0489 & 0.0467 & 0.0652 & 0.0598 & 0.0571 \\
    & 700  & 0.0495 & 0.0448 & 0.0426 & 0.0598 & 0.0549 & 0.0525 \\
    & 800  & 0.0458 & 0.0415 & 0.0395 & 0.0556 & 0.0508 & 0.0485 \\
    & 900  & 0.0429 & 0.0388 & 0.0368 & 0.0521 & 0.0475 & 0.0452 \\
    & 1000 & 0.0405 & 0.0365 & 0.0346 & 0.0492 & 0.0448 & 0.0426 \\
\hline
2.0 & 100  & 0.1256 & 0.1148 & 0.1085 & 0.1485 & 0.1362 & 0.1295 \\
    & 200  & 0.0864 & 0.0792 & 0.0745 & 0.1054 & 0.0968 & 0.0915 \\
    & 300  & 0.0695 & 0.0632 & 0.0598 & 0.0856 & 0.0785 & 0.0742 \\
    & 400  & 0.0592 & 0.0538 & 0.0505 & 0.0728 & 0.0665 & 0.0628 \\
    & 500  & 0.0524 & 0.0475 & 0.0448 & 0.0645 & 0.0589 & 0.0556 \\
    & 600  & 0.0475 & 0.0428 & 0.0402 & 0.0582 & 0.0532 & 0.0501 \\
    & 700  & 0.0436 & 0.0392 & 0.0368 & 0.0535 & 0.0488 & 0.0462 \\
    & 800  & 0.0405 & 0.0365 & 0.0342 & 0.0495 & 0.0452 & 0.0428 \\
    & 900  & 0.0378 & 0.0340 & 0.0318 & 0.0464 & 0.0422 & 0.0398 \\
    & 1000 & 0.0356 & 0.0319 & 0.0299 & 0.0438 & 0.0398 & 0.0375 \\
\hline
2.5 & 100  & 0.1345 & 0.1228 & 0.1165 & 0.1568 & 0.1452 & 0.1385 \\
    & 200  & 0.0925 & 0.0856 & 0.0805 & 0.1124 & 0.1028 & 0.0982 \\
    & 300  & 0.0742 & 0.0678 & 0.0642 & 0.0905 & 0.0825 & 0.0788 \\
    & 400  & 0.0635 & 0.0575 & 0.0545 & 0.0772 & 0.0705 & 0.0668 \\
    & 500  & 0.0562 & 0.0508 & 0.0482 & 0.0685 & 0.0622 & 0.0589 \\
    & 600  & 0.0508 & 0.0458 & 0.0435 & 0.0618 & 0.0562 & 0.0532 \\
    & 700  & 0.0468 & 0.0420 & 0.0398 & 0.0568 & 0.0515 & 0.0488 \\
    & 800  & 0.0435 & 0.0390 & 0.0368 & 0.0526 & 0.0478 & 0.0452 \\
    & 900  & 0.0408 & 0.0365 & 0.0345 & 0.0492 & 0.0448 & 0.0422 \\
    & 1000 & 0.0385 & 0.0342 & 0.0324 & 0.0465 & 0.0422 & 0.0398 \\
\hline
\end{tabular*}
\end{table}

In the applications, when the observed statistic $T_{N,k}^{\mathrm{KL}} \ge t_{0.05}$ is satisfied, the multivariate normality is rejected at the $\alpha=0.05$ level. In this case, $t_{0.05}$ is calculated using bootstrap calibration (Algorithm~\ref{alg:KL_bootstrap}) or obtained from Table~\ref{table:crt_val_KL} for the settings reported here.

\section{Conclusion}
\label{sec:conclusion}

In this paper, we propose an information theory based framework that provides a computationally simple method for multivariate goodness-of-fit testing in environments in which the likelihood density estimate is unreliable. 

We developed a KL-decomposition approach for a multivariate normality test using the k-nearest neighbor estimate of Shannon entropy. The fundamental structural feature is that, within a fixed mean-covariance class, the Gaussian reference value converts the KL decomposition into an entropy difference, such that the statistic $T_{N,k}^{\mathrm{KL}}$ predicts only a non-negative inconsistency that vanishes under multivariate normality. Under standard regularity conditions, the kNN entropy estimator is consistent, which directly translates to the consistency of the resulting KL-based test statistic.

Our numerical study provides three practical results. First,
statistics in the Gaussian benchmark rapidly approach zero as $N$ increases.
Secondly, the power increases smoothly with deviation power, including shape changes in generalized Gaussian models and heavy-tailed Student-type alternatives.
Finally, bootstrap calibration provides stable finite sample control
and generates critical values that behave predicted by $N$, $m$ and $k$.


\clearpage
\bibliographystyle{plain} 
\bibliography{KL_Shannon} 

\end{document}